\documentclass[reqno,a4paper, 11pt]{amsart}

\usepackage[a4paper=true,pdfpagelabels]{hyperref}
\usepackage{graphicx}

\usepackage[ansinew]{inputenc}
\usepackage{amsfonts,epsfig}
\usepackage{latexsym}
\usepackage{mathabx}
\usepackage{amsmath}
\usepackage{amssymb}

\usepackage{color}

\newtheorem{theorem}{Theorem}
\newtheorem{lemma}[theorem]{Lemma}

\newtheorem{proposition}[theorem]{Proposition}

\newtheorem{lettertheorem}{Theorem}
\newtheorem{letterlemma}[lettertheorem]{Lemma}

\theoremstyle{definition}

\theoremstyle{remark}

\numberwithin{equation}{section}

\newcommand{\abs}[1]{\lvert#1\rvert}
\newcommand{\nm}[1]{\lVert#1\rVert}

\newcommand{\B}{\mathcal{B}}
\newcommand{\D}{\mathbb{D}}
\newcommand{\DD}{\widehat{\mathcal{D}}}
\newcommand{\Dd}{\widecheck{\mathcal{D}}}

\newcommand{\DDD}{\mathcal{D}}

\newcommand{\N}{\mathbb{N}}

\newcommand{\RR}{\mathbb{R}}

\newcommand{\C}{\mathbb{C}}

\newcommand{\e}{\varepsilon}

\renewcommand{\phi}{\varphi}
\newcommand{\A}{\mathcal{A}}

\newcommand{\T}{\mathbb{T}}

\def\BMOA{\mathord{\rm BMOA}}

\def\HL{\mathord{\rm HL}}

           \def\e{\varepsilon}
     \def\om{\omega}      
       \def\t{\theta}       
                  \def\z{\zeta}

\DeclareMathOperator{\supp}{supp}
\renewcommand{\H}{\mathcal{H}}

\begin{document}
\title[Integral operators mapping into the space of bounded analytic functions]{Integral operators induced by symbols with non-negative Maclaurin coefficients mapping into $H^\infty$}



\author{Jos\'e \'Angel Pel\'aez}
\address{Departamento de An\'alisis Matem\'atico, Universidad de M\'alaga, Campus de
Teatinos, 29071 M\'alaga, Spain} \email{japelaez@uma.es}

\author{Jouni R\"atty\"a}
\address{University of Eastern Finland, P.O.Box 111, 80101 Joensuu, Finland}
\email{jouni.rattya@uef.fi}

\author{Fanglei Wu}
\address{Department of Mathematics, Shantou University, Shantou, Guangdong 515063, China}
\address{University of Eastern Finland, P.O.Box 111, 80101 Joensuu, Finland}
\email{fangleiwu1992@gmail.com}

\thanks{The research of the first author was supported in part by Ministerio de Econom\'{\i}a y Competitividad, Spain, projects
PGC2018-096166-B-100; La Junta de Andaluc{\'i}a,
projects FQM210  and UMA18-FEDERJA-002}

\begin{abstract}
For analytic functions $g$ on the unit disc with non-negative Maclaurin coefficients, we describe the boundedness and compactness 
of the integral operator $T_g(f)(z)=\int_0^zf(\z)g'(\z)\,d\z$  from a space $X$ of analytic functions in the unit disc to $H^\infty$, in terms 
of neat and useful conditions  on the Maclaurin coefficients of $g$. The choices of $X$ that will be considered contain the Hardy and the Hardy-Littlewood spaces, the Dirichlet-type spaces $D^p_{p-1}$, as well as the classical Bloch and $\BMOA$ spaces. 
\end{abstract}

\maketitle

\section{Introduction and main results}

Let $\H(\D)$ denote the space of analytic functions in the unit disc $\D=\{z\in\C:|z|<1\}$. Each $g\in\H(\D)$ induces the integral operator defined by
	$$
	T_g(f)(z)=\int_0^zf(\zeta)g'(\zeta)\,d\zeta,\quad z\in\D.
	$$
The question of when this operator is either bounded or compact has been extensively studied in a large variety of spaces of analytic functions since the appearance of the seminal works, related to the Hardy and Bergman spaces, due to Aleman, Cima, Pommerenke and Siskakis~\cite{AC,AS,Pom}. Getting neat conditions on the symbol $g$ which describe the bounded and compact operators $T_g$ acting from a Banach space $X\subset\H(\D)$ to the Hardy space $H^\infty$ is known to be a tough problem~\cite{AJS,CPPR,SSV2017}. However, recently an abstract approach to the study of this question was given in \cite{CPPR}. One of the basic results there is the reproducing kernel dual testing condition provided in \cite[Theorem~2.2]{CPPR}.
It states that, if $X^\star\simeq Y$ via the $H(\beta)$-pairing
	$$
	\langle f,g\rangle_{H(\beta)}=\lim_{r\to1^-}\sum_{n=0}^\infty\widehat{f}(n)\overline{\widehat{g}(n)}\beta_nr^n,
	\quad f(z)=\sum_{n=0}^\infty\widehat{f}(n)z^n,\quad g(z)=\sum_{n=0}^\infty\widehat{g}(n)z^n,
	$$
where $\lim_{n\to\infty}\sqrt[n]{\beta_n}=1$, then $T_g:X\to H^\infty$ is bounded if and only if $\sup_{z\in\D}\left\|G^{H(\beta)}_{g,z}\right\|_Y<\infty$, where
		$$
		\overline{G^{H(\beta)}_{g,z}(w)}
		=\int_0^zg'(\zeta)\overline{K_\zeta^{H(\beta)}(w)}\,d\zeta
		=\overline{T_g^*(K_z^{H(\beta)})(w)},\quad z\in\D,\quad w\in\overline{\D},
		$$
and $K_\zeta^{H(\beta)}$ are the reproducing kernels of the Hilbert space $H(\beta)$.
 
Theoretically, the above relatively simple result offers a characterization of the boundedness in the case of most of the natural spaces one can think of. However, if one tries to apply this in praxis one observes that it looks like a reformulation rather than a solution of the problem. This is due to the fact that treating the function $G^{H(\beta)}_{g,z}$ in the dual space of $X$ is often laborious if not even frustrating. Because of these reasons, in this study we restrict ourselves to the case in which the symbol $g$ has non-negative Maclaurin coefficients, and search for neat and useful conditions in terms of the Maclaurin coefficients of $g$ that can be used to test if $T_g$ is either bounded or compact from $X$ to $H^\infty$. The starting point is the characterization \cite[Theorem~2.2]{CPPR} given above, and the choices for $X$ that will be considered in the sequel contain the Hardy and the Hardy-Littlewood spaces, and certain Dirichlet-type spaces, as well as the classical Bloch space and $\BMOA$. Next, the main findings of this study along with necessary definitions are stated.

For $0<p\le\infty$, the Hardy space $H^p$ consists of $f\in\H(\D)$ for which
    \begin{equation*}\label{normi}
    \|f\|_{H^p}=\sup_{0<r<1}M_p(r,f)<\infty,
    \end{equation*}
where
    $$
    M_p(r,f)=\left (\frac{1}{2\pi}\int_0^{2\pi}
    |f(re^{i\theta})|^p\,d\theta\right )^{\frac{1}{p}},\quad 0<p<\infty,
    $$
and
    $$
    M_\infty(r,f)=\max_{0\le\theta\le2\pi}|f(re^{i\theta})|.
    $$
Further, $f\in\H(\D)$ belongs to the Dirichlet-type space $D^p_{p-1}$ if
	$$
	\|f\|^p_{D^p_{p-1}}=\int_\D|f'(z)|^p(1-|z|)^{p-1}\,dA(z)+|f(0)|^p<\infty,
	$$
where $dA(z)=\frac{dx\,dy}{\pi}$ is the normalized area measure on $\D$. The closely related Hardy-Littlewood space $\HL_p$ contains those $f\in\H(\D)$ whose Maclaurin coefficients $\{\widehat{f}(n)\}_{n=0}^\infty$ satisfy
	$$
	\nm{f}^p_{\HL_p}=\sum_{n=0}^{\infty}\abs{\widehat{f}(n)}^p (n+1)^{p-2}<\infty.
	$$
These spaces satisfy the well-known inclusions
	\begin{equation}\label{Hardy1}
	D^p_{p-1}\subset H^p\subset\HL_p,\quad 0<p\le 2,
	\end{equation}
and
	\begin{equation}\label{Hardy2}
	\HL_p\subset H^p\subset D^p_{p-1} ,\quad 2\le p<\infty,
	\end{equation}
by \cite{D,Flett,LP}. Each of these inclusions is strict unless $p=2$, in which case all the spaces are the same by direct calculations or straightforward applications of Parseval's formula and Green's theorem. 

Our first main result reveals that $T_g$ does not distinguish $H^p$, $\HL_p$ and $D^p_{p-1}$ when it acts boundedly or compactly from one of these spaces to $H^\infty$, provided $1<p<\infty$ and the symbol $g$ has non-negative Maclaurin coefficients. Here, as usual, the conjugate index of $1<p<\infty$ is the number $p'$ such that $\frac1p+\frac1{p'}=1$.

\begin{theorem}\label{Theorem:Hardy-p<2}
Let $1<p<\infty$ and $g\in H^\infty$ such that $\widehat{g}(n)\geq0$ for all $n\in\mathbb{N}\cup\{0\}$. Further, let $X_p\in\{H^p,D^p_{p-1}\HL_p\}$. Then $T_g:~X_p\rightarrow H^\infty$ is bounded (equivalently compact) if and only if
	\begin{equation}\label{pipeli}
	\sum_{k=0}^{\infty}(k+1)^{p'-2}\left(\sum_{n=0}^{\infty}\frac{(n+1)\widehat{g}(n+1)}{n+k+1}\right)^{p'}<\infty.
	\end{equation}
Moreover,
	\begin{equation}\label{eq:j11}
	\|T_g\|_{X_p\rightarrow H^\infty}^{p'}
	\asymp\sum_{k=0}^{\infty}(k+1)^{p'-2}\left(\sum_{n=0}^{\infty}\frac{(n+1)\widehat{g}(n+1)}{n+k+1}\right)^{p'}.
	\end{equation}
\end{theorem}

On the way to Theorem~\ref{Theorem:Hardy-p<2} we show in Theorem~\ref{th:dpsuf} below that for each $g\in\H(\D)$ we have
	\begin{equation}\label{poi}
	\|T_g\|_{X_p\rightarrow H^\infty}^{p'}
  \lesssim
  \sum_{k=0}^{\infty}(k+1)^{p'-2}\left(\sum_{n=0}^{\infty}\frac{(n+1)|\widehat{g}(n+1)|}{n+k+1}\right)^{p'}
	\end{equation}
for each $X_p\in\{H^p,D^p_{p-1}\HL_p\}$. This observation offers a relatively easy way to see if a given general symbol $g$ induces a bounded operator on the Hardy space $H^p$ with $1<p<\infty$.

The proof of Theorem~\ref{Theorem:Hardy-p<2} relies on \cite[Theorem~2.2]{CPPR}, and duality relations for $H^p$, $D^p_{p-1}$ and $\HL_p$ with $1<p<\infty$. The dual of $H^p$ is of course isomorphic to $H^{p'}$ via the $H^2$-pairing (the Cauchy-pairing), and certainly many experts working on the field now that $(D^p_{p-1})^\star\simeq D^{p'}_{p'-1}$ and $(\HL_p)^\star\simeq\HL_{p'}$ via the same pairing. Since we do not know exact references for the last-mentioned two dualities, we give proofs in Section~\ref{Sec:dualities} where also other less obvious duality relations are treated. Another tool that we will employ in the proof of Theorem~\ref{Theorem:Hardy-p<2} is of technical nature, and concerns smooth universal Ces\'aro basis of polynomials~\cite[Section 5.2]{Pavlovic}. The proof of Theorem~\ref{Theorem:Hardy-p<2} is presented in Section~\ref{Sec:pili}. 

If $0<p<1$, then $T_g:H^p\to H^\infty$ is bounded if and only if $g$ is a constant by~\cite[Theorem~2.5(vi)]{CPPR}. Further, by~\cite[Theorem~4.2(ii)]{CPPR}, $T_g:H^1\to H^\infty$ is compact if and only if $g$ is a constant. Therefore the same conclusions are valid for $T_g$ acting on $\HL_p$ by \eqref{Hardy1}. The following result clarifies the situation with regard to the Dirichlet-type spaces $D^p_{p-1}$. Here and from now on $T(X,H^\infty)$ (resp. $T_c(X,H^\infty)$) denotes the set of $g\in\H(\D)$ such that $T_g:X\to H^\infty$ is bounded (resp. compact).

\begin{theorem}\label{Theorem:p=1}
Let $g\in H^\infty$ and $X_p\in\{H^p,D^p_{p-1},\HL_p\}$. Then the following assertions hold:
\begin{itemize}
\item[\rm(i)] $T(X_1,H^\infty)$ contains all polynomials;
\item[\rm(ii)] $T(X_p,H^\infty)$ consists of constant functions only if $0<p<1$;
\item[\rm(iii)] $T_c(X_1,H^\infty)$ consists of constant functions only.
\end{itemize}
\end{theorem}

In the proof of Theorem~\ref{Theorem:p=1} we use identifications of the duals of $\HL_1$ and $D^p_{p-1}$ with $0<p\le1$. 
Since many dual spaces $X^\star$ can be described, via the $H^2$-pairing,  as  the space of coefficient multipliers from $X$ to the disc algebra  \cite[Proposition~1.3]{P}, a natural characterization of the dual of $\HL_1$ is easy to find by using the relation $(\ell^1)^\star\simeq\ell^\infty$. We do this in Section~\ref{Sec:dualities} when we prove Lemma~\ref{le:dualHL1} which states that $(\HL_1)^\star\simeq\HL_{\infty}$ via the $H^2$-pairing with equivalence of norms. The space $\HL_{\infty}$ consists of $f\in\H(\D)$ such that its Maclaurin coefficients $\{\widehat{f}(n)\}_{n=0}^{\infty}$ satisfy
	$$
	\|f\|_{\HL_{\infty}}=\sup_{n\in\N\cup\{0\}}\left(|\widehat{f}(n)|(n+1)\right)<\infty.	
	$$
To find a suitable dual of $D^p_{p-1}$ with $0<p\le1$ is not that straightforward. Lemma~\ref{lemma:dual-Dirichlet-2} in Section~\ref{Sec:dualities} states that $(D^p_{p-1})^\star\simeq\B^2$ via the $A^2_{\frac1p-1}$-pairing.	Here $A^2_\beta$ refers to the Bergman Hilbert space induced by the standard weight $(\beta+1)(1-|z|^2)^\beta$. Further, for $0<\alpha<\infty$ and $f\in\H(\D)$, the $\alpha$-Bloch space $\B^\alpha$ consists of $f\in\H(\D)$ such that
	$$
	\|f\|_{\B^\alpha}=\sup_{z\in\D}|f'(z)|(1-|z|^2)^\alpha+|f(0)|<\infty.
	$$
The proof of the duality relation $(D^p_{p-1})^\star\simeq\B^2$ is lengthy, and apart from standard tools, such as Green's theorem and continuous embeddings between different weighted Bergman spaces, it also relies on a use of  smooth universal Ces\'aro basis of polynomials. The last-mentioned creatures are used to show that a certain function, dependent of $p$, is a coefficient multiplier of $\B^2$. 

The next result is the counterpart of Theorem~\ref{Theorem:Hardy-p<2} in the case $p=1$. Observe that \eqref{Eq:p=1} is the limit case $p'=\infty$ of \eqref{eq:j11}, and that the supremum there is in fact the limit as $k\to\infty$ since the quantity over which the supremum is taken is increasing in $k$.

\begin{theorem}\label{Theorem:p=1'}
Let $g\in H^\infty$ such that $\widehat{g}(n)\geq0$ for all $n\in\mathbb{N}\cup\{0\}$, and $X_1\in\{H^1,D^1_{0},\HL_1\}$. Then $T_g:X_1\rightarrow H^\infty$ is bounded if and only if
	\begin{equation}\label{Eq:p=1}
	\sup_{k\in\N\cup\{0\}}\left((k+1)\sum_{n=0}^{\infty}\frac{(n+1)\widehat{g}(n+1)}{n+k+1}\right)<\infty.
	\end{equation}
Moreover,
	\begin{equation}\label{eq:j11'}
	\|T_g\|_{X_1\rightarrow H^\infty}
	\asymp\sup_{k\in\N\cup\{0\}}\left((k+1)\sum_{n=0}^{\infty}\frac{(n+1)\widehat{g}(n+1)}{n+k+1}\right).
	\end{equation}
\end{theorem}

Theorem~\ref{Theorem:p=1'} is relatively straightforward to establish once the tools needed for Theorem~\ref{Theorem:p=1} are on the table. 
Both of these theorems are proved in Section~\ref{Sec:2-3}.

Our last result concerns the case when $T_g$ acts from the Bloch space or $\BMOA$ to $H^\infty$. Recall that the classical Bloch space $\B$ is just the space $\B^1$ defined before Theorem~\ref{Theorem:p=1'}. Further, let
	$$
	\|f\|_{H^\infty_{\log}}=\sup_{z\in\D}|f(z)|\left(\log\frac{e}{1-|z|}\right)^{-1},
	$$
and recall that $\BMOA$ consists of the functions in the Hardy space $H^1$ that have
bounded mean oscillation on the boundary $\T$. The space $\BMOA$ can be equipped with several different norms~\cite{Daniel}. We will use the one given by
    $$
    \|g\|^2_{\BMOA}=\sup_{a\in\D}\frac{\int_{S(a)}|g'(z)|^2(1-|z|^2)\,dA(z)}{1-|a|}+|g(0)|^2,
    $$
where $S(a)=\{\z:1-|a|<|\z|<1,\,|\arg\z-\arg a|<(1-|a|)/2\}$ is the Carleson square induced by $a\in\D\setminus\{0\}$ and $S(0)=\D$.
It is well-known that
	\begin{equation}\label{Eq:embedding}
	\|f\|_{H^\infty_{\log}}\lesssim\|f\|_\B\lesssim\|f\|_{\BMOA}\lesssim\|f\|_{H^\infty},\quad f\in\H(\D).
	\end{equation}

Our last main result says that $T_g$ does not distinguish $\BMOA$, $\B$ and $H^\infty_{\log}$ when it acts boundedly or compactly from one of these spaces to $H^\infty$, if the symbol $g$ has non-negative Maclaurin coefficients.

\begin{theorem}\label{Bounded}
Let $X\subset\H(\D)$ be a Banach space such that $\BMOA\subset X\subset H^\infty_{\log}$ and
let $g\in H^{\infty}$ with $\widehat{g}(n)\geq0$ for all $n\in\mathbb{N}\cup\{0\}$. Then the following statements are equivalent:
\begin{itemize}
\item[(i)] $T_g:~H_{\log}^\infty\to H^\infty$ is bounded (equivalently compact);
\item[(ii)] $T_g:~X\to H^\infty$ is bounded (equivalently compact);
\item[(iii)] $T_g:~\BMOA\to H^\infty$ is bounded (equivalently compact);
\item[(iv)] $\sum_{n=0}^\infty\widehat{g}(n+1)\log(n+2)<\infty$;
\item[(v)] $\int_0^1M_{\infty}(r,g')\log\frac{e}{1-r}dr<\infty$.
\end{itemize}
Moreover,
	\begin{equation}\label{Eq:norm-estimates}
	\begin{split}
	\|T_g\|_{\BMOA\rightarrow H^\infty}
	&\asymp\|T_g\|_{X\rightarrow H^\infty}
	\asymp\|T_g\|_{H_{\log}^\infty\rightarrow H^\infty}
	\asymp\int_0^1M_\infty(r,g')\log\frac{e}{1-r}dr\\
	&\asymp\sum_{n=0}^\infty\widehat{g}(n+1)\log(n+2).
	\end{split}
	\end{equation}
\end{theorem}

The proof of Theorem~\ref{Bounded}, given in Section~\ref{Sec:Bloch}, reveals that 
	\begin{equation*}
	\|T_g\|_{H_{\log}^\infty\rightarrow H^\infty}
	\lesssim\int_0^1M_\infty(r,g')\log\frac{e}{1-r}dr
	\lesssim\sum_{n=0}^\infty|\widehat{g}(n+1)|\log(n+2)
	\end{equation*}
for each $g\in\H(\D)$. The hypothesis on the coefficients is only used when the right most quantity above is dominated by the operator norm.

Probably the most obvious election for the space $X$ in the statement of Theorem~\ref{Bounded} is the classical Bloch space $\B$. However, there are other choices for $X$ which arise naturally in the theory of  integral operators, see Section~\ref{Sec:Bloch} for further details.

It is worth mentioning that the hypotheses $g\in H^\infty$ in  Theorems~\ref{Theorem:Hardy-p<2}-\ref{Bounded} is not a restriction, because it is an obvious necessary condition so that $T_g: X\to H^\infty$ is bounded.

To this end, couple of words about the notation used in this paper. The letter $C=C(\cdot)$ will denote an absolute constant whose value depends on the parameters indicated
in the parenthesis, and may change from one occurrence to another. If there exists a constant
$C=C(\cdot)>0$ such that $a\le Cb$, then we write either $a\lesssim b$ or $b\gtrsim a$. In particular, if $a\lesssim b$ and
$a\gtrsim b$, then we denote $a\asymp b$ and say that $a$ and $b$ are comparable.

\section{Dualities}\label{Sec:dualities}

In this section we will discuss the duality relations employed to prove the main results of the paper. Apart from the well-known relation $(H^p)^\star\simeq H^{p'}$, $1<p<\infty$, we will need to know the dual spaces, with respect to appropriate pairings, of $D^p_{p-1}$ and $\HL_p$ for $0<p\le 1$ and $1<p<\infty$, respectively. 
 
The following lemma describes the dual of the Dirichlet-type space $D^p_{p-1}$ when $1<p<\infty$, and it will be needed in the proof of Theorem~\ref{Theorem:Hardy-p<2}. We believe that the result itself must be known at least by experts working on the field, but since we do not know an exact reference, we include a proof here.

\begin{lemma}\label{le:dualdp}
Let $1<p<\infty$. Then $(D^p_{p-1})^\star\simeq D^{p'}_{p'-1}$ via the $H^2$-pairing with equivalence of norms.
\end{lemma}

\begin{proof}
Let us first show that each $g\in D^{p'}_{p'-1}$ induces a bounded linear functional on $D^p_{p-1}$. Green's theorem implies
	\begin{equation}\label{Green}
	\langle f,g\rangle_{H^2}=2\int_{\D}f'(\z)\overline{g'(\z)}\log\frac{1}{|\z|}\,dA(\z)+f(0)\overline{g(0)},
	\end{equation}
from which H\"{o}lder's inequality yields
	\begin{equation*}\begin{split}
|\langle f,g\rangle_{H^2}|
 &\lesssim \int_{\D}  |f'(\z)||g'(\z)|(1-|\z|)\,dA(\z)+|f(0)||g(0)|
\lesssim\|f\|_{D^p_{p-1}}\|g\|_{D^{p'}_{p'-1}},\quad f,g\in\H(\D),
\end{split}
	\end{equation*}
where the first step is an easy consequence of the inequality $-\log t\le\frac1t(1-t)$, valid for all $0<t\le1$, and the monotonicity of $M_p(r,h)$ for each $0<p<\infty$ and $h\in\H(\D)$.
Thus each $g\in D^{p'}_{p'-1}$ induces a bounded linear functional on $D^p_{p-1}$ via the $H^2$-pairing.

Let now $L$ be a bounded linear functional on $D^p_{p-1}$. Consider the weights $\om(z)=-2\log|z|$ and $\nu(z)=\left(-2\log|z|\right)^\frac1{p-1}$, defined in the punctured unit disc. The proof of \cite[Theorem~3]{PR2014} now shows that the Bergman projection $P_\om$, induced by $\om$, is bounded from $L^{p'}_\nu$ into itself because the weight $\left(\frac{\om}{\nu}\right)^p\nu=\om^{p-1}$ is sufficiently smooth. It then follows from the proof of \cite[Theorem~6]{PR2014} and standard arguments that $(A^p_{p-1})^\star\simeq A^{p'}_{p'-1}$ under the pairing
	$$
	\langle f,g\rangle_{A^2_{\omega}}=2\int_{\D}f(z)\overline{g(z)}\log\frac{1}{|z|}\,dA(z).
	$$
We note that this duality relation of the weighted Bergman spaces is essentially contained in \cite[Theorem~2.1]{L} as a special case, but with respect to a slightly different pairing. The method there would certainly work also in our setting and therefore offers an alternative way to deduce this duality. Getting back to the proof of the lemma, we observe that, for each $f\in D^p_{p-1}$, there exists $F=F_f\in A^p_{p-1}$ such that $\mathcal{I}(F)=f-f(0)$, where $\mathcal{I}(F)(z)=\int_0^zF(\zeta)\,d\zeta$. Further, $\mathcal{I}$ is an isometric mapping from $A^p_{p-1}$ to $D^p_{p-1}$, in particular, it is bounded. Therefore the composition $L\circ\mathcal{I}$ is a bounded linear functional on $A^p_{p-1}$, and hence there exists a unique $G\in A^{p'}_{p'-1}$ such that $\|G\|_{A^{p'}_{p'-1}}\lesssim\|L\circ\mathcal{I}\|\lesssim\|L\|$ and
	\begin{equation*}
	\begin{split}
	L(f)=L(f-f(0)+f(0))&=(L\circ\mathcal{I})(F)+f(0)L(1)\\
	&=2\int_{\D}F(z)\overline{G(z)}\log\frac{1}{|z|}dA(z)+f(0)L(1)\\
	&=2\int_{\D}f'(z)\overline{G(z)}\log\frac{1}{|z|}dA(z)+f(0)L(1).
	\end{split}
	\end{equation*}
Further, since $G\in A^{p'}_{p'-1}$, there exists a unique $g\in D^{p'}_{p'-1}$ such that $g'=G$ and $\overline{g(0)}=L(1)$. Consequently, there exists a unique $g\in D^{p'}_{p'-1}$ such that
	$$
	L(f)=2\int_{\D}f'(z)\overline{g'(z)}\log\frac{1}{|z|}dA(z)+f(0)\overline{g(0)}=\langle f,g\rangle_{H^2},
	$$
where the last identity follows from \eqref{Green}. Moreover, $\|g\|^{p'}_{D^{p'}_{p'-1}}=\|G\|^{p'}_{A^{p'}_{p'-1}}+|L(1)|^{p'}\lesssim\|L\|^{p'}$, and the assertion is proved.
\end{proof}

To prove Theorems~\ref{Theorem:p=1} and~\ref{Theorem:p=1'} we need to know the dual of $D^p_{p-1}$ with $0<p\le1$. In order to do that, some more notation is needed. For $0<\alpha<\infty$, the space $H^\infty_\alpha$ consists of $f\in\H(\D)$ such that
	$$
	\|f\|_{H^\infty_\alpha}=\sup_{z\in\D}|f(z)|(1-|z|^2)^\alpha<\infty.
	$$
It is well-known that 
	\begin{equation}\label{Eq-alpha-bloch}
	\|f\|_{H^\infty_\alpha}\asymp\|f\|_{\B^{\alpha+1}},\quad f\in\H(\D),
	\end{equation}
for each $0<\alpha<\infty$. 

We will also need background on certain smooth polynomials defined in terms of Hadamard products. Recall that the Hadamard product of $f\in\H(\D)$ and $g\in\H(\D)$ is formally defined as
    $$
    (f\ast g)(z)=\sum_{k=0}^\infty\widehat{f}(k)\widehat{g}(k)z^k,\quad z\in\D.
    $$
A direct calculation shows that
    \begin{equation}\label{eq:hadprod}
    (f\ast g)(r^2e^{it})
    =\frac{1}{2\pi}\int_{-\pi}^\pi f(re^{i(t+\theta)})g(re^{-i\t})\,d\t.
    \end{equation}
If  $W(z)=\sum_{k\in J}b_kz^k$ is a polynomial and $f\in\H(\D)$, then the Hadamard product
    $$
    (W\ast f)(z)=\sum_{k\in J}b_k\widehat{f}(k)z^k
    $$
is well defined. Further, if $\Phi:\mathbb{R}\to\C$ is a $C^\infty$-function with compact support $\supp(\Phi)$ in $(0,\infty)$, set
    $$
    A_{\Phi,m}=\max_{s\in\mathbb{R}}|\Phi(s)|+\max_{s\in\mathbb{R}}|\Phi^{(m)}(s)|,
    $$
and consider the polynomials
   \begin{equation}\label{eq:polynomials}
    W_n^\Phi(z)=\sum_{k\in\mathbb
    Z}\Phi\left(\frac{k}{n}\right)z^{k},\quad n\in\N.
    \end{equation}
With this notation we can state the next auxiliary result that follows by \cite[p.~111--113]{Pavlovic}.

\begin{lettertheorem}\label{th:cesaro}
Let $\Phi:\mathbb{R}\to\C$ be a $C^\infty$-function such that
$\supp(\Phi)\subset(0, \infty)$ is compact. Then for each $p\in(0,\infty)$ and $m\in\N$ with $mp>1$, there exists a constant
$C=C(p)>0$ such that
    $$
    \|W_N^\Phi\ast f\|_{H^p}\le C A_{\Phi,m}\|f\|_{H^p}
    $$
for all $f\in H^p$ and $N\in\N$.
\end{lettertheorem}

Theorem~\ref{th:cesaro} shows that the polynomials $\{W_n^\Phi\}_{n\in\N}$ can be seen as a universal C\'esaro basis for $H^p$ for any $0<p<\infty$. A particular case of the previous construction is useful for our purposes. By following \cite[Section~2]{JevPac98}, let $\Psi:\mathbb{R}\to\mathbb{R}$ be a $C^\infty$-function such that
    \begin{enumerate}
    \item $\Psi\equiv1$ on $(-\infty,1]$,
    \item $\Psi\equiv0$ on $[2,\infty)$,
    \item $\Psi$ is decreasing and positive on $(1,2)$,
    \end{enumerate}
and set $\psi(t)=\Psi\left(\frac{t}{2}\right)-\Psi(t)$ for all $t\in\mathbb{R}$. Let $V_0(z)=1+z$
and
    \begin{equation}\label{vn}
    V_n(z)=W^\psi_{2^{n-1}}(z)=\sum_{k=0}^\infty
    \psi\left(\frac{k}{2^{n-1}}\right)z^k=\sum_{k=2^{n-1}}^{2^{n+1}-1}
    \psi\left(\frac{k}{2^{n-1}}\right)z^k,\quad n\in\N.
    \end{equation}
These polynomials have the following properties with regard to
smooth partial sums, see \cite[p. 175--177]{JevPac98} or \cite[p. 143--144]{Pabook2} for details:
    \begin{equation}
    \begin{split}\label{propervn}
    f(z)&=\sum_{n=0}^\infty (V_n\ast f)(z),\quad f\in\H(\D),\\
    \|V_n\ast f\|_{H^p}&\le C\|f\|_{H^p},\quad f\in H^p,\quad 0<p<\infty,\\
    \|V_n\|_{H^p}&\asymp 2^{n(1-1/p)}, \quad 0< p<\infty.
    \end{split}
    \end{equation}

With these preparations we can describe the dual of $D^p_{p-1}$ with $0<p\le1$.

\begin{lemma}\label{lemma:dual-Dirichlet-2}
Let $0<p\le1$. Then $(D^p_{p-1})^\star\simeq\B^2$ via the $A^2_{\frac1p-1}$-pairing with equivalence of norms.
\end{lemma}	

\begin{proof}
Let $f,g\in\H(\D)$. Then Green's formula and Fubini's theorem yield
	\begin{equation*}
	\begin{split}
	|\langle f,g\rangle_{A^2_{\frac1p-1}}|
	&=\left|\int_\D f(z)\overline{g(z)}(1-|z|^2)^{\frac1p-1}\,dA(z)\right|\\
	&= \left| \int_0^1 \left(2\int_\D f'(rz)r\overline{g'(rz)r}\log\frac1{|z|}\,dA(z)+f(0)\overline{g(0)}\right)
	(1-r^2)^{\frac1p-1}r\,dr\right|
    \\
	&\le2\int_\D|f'(\z)|g'(\z)|\left(\int_{|\z|}^1\log\frac{r}{|\z|}(1-r^2)^{\frac1p-1}r\,dr\right)\,dA(\z)+|f(0)||g(0)|.
	\end{split}
	\end{equation*}
The inequality $-\log t\le\frac1t(1-t)$, valid for all $0<t\le1$, now gives
	$$
	\int_{|\z|}^1\log\frac{r}{|\z|}(1-r^2)^{\frac1p-1}r\,dr
	\le2^\frac1p\log\frac{1}{|\z|}\int_{|\z|}^1(1-r)^{\frac1p-1}\,dr
	\le p2^\frac1p\frac{(1-|\z|)^{1+\frac1p}}{|\z|},\quad \z\in\D\setminus\{0\}.
	$$
By using this and the continuous embedding $A^{p}_{p-1}\subset A^1_{\frac1p-1}$, valid for $0<p\le1$ by \cite[Theorem~1]{L2}, we deduce
	\begin{equation*}
	\begin{split}
	|\langle f,g\rangle_{A^2_{\frac1p-1}}|
	&\le p2^\frac1p\int_\D|f'(\z)||g'(\z)|\frac{(1-|\z|)^{1+\frac1p}}{|\z|}\,dA(\z)+|f(0)||g(0)|\\
	&\lesssim\|g\|_{\B^2}\int_\D|f'(\z)|(1-|\z|)^{\frac1p-1}\,dA(\z)+|f(0)||g(0)|
	\lesssim\|g\|_{\B^2}\|f\|_{D^p_{p-1}},
	\end{split}
	\end{equation*}
and hence each $g\in\B^2$ induces a bounded linear functional on $D^p_{p-1}$ via the $A^2_{\frac1p-1}$-pairing.

Let $L\in(D^p_{p-1})^\star$, and recall that $\mathcal{I}(F)(z)=\int_0^zF(\zeta)\,d\zeta$. Then $|(L\circ\mathcal{I})(F)|\lesssim\|\mathcal{I}(F)\|_{D^p_{p-1}}=\|F\|_{A^p_{p-1}}$ for all $F\in A^p_{p-1}$. Therefore $L\circ\mathcal{I}\in(A^p_{p-1})^\star$. Since $(A^p_{p-1})^\star$ is isomorphic to the Bloch space via the $A^2_{\frac1p-1}$-pairing by \cite[Theorem~A]{Zhu2}, there exists a unique $G\in\B$ such that $\|G\|_\B\lesssim\|L\circ\mathcal{I}\|\lesssim\|L\|$ and $(L\circ\mathcal{I})(F)=\langle F,G\rangle_{A^2_{\frac1p}-1}$ for all $F\in A^p_{p-1}$. For each $f\in D^p_{p-1}$, there exists $F=F_f\in A^p_{p-1}$ such that $\mathcal{I}(F)=f-f(0)$. Therefore 
	\begin{equation*}
	\begin{split}
	L(f)&=L(f-f(0)+f(0))=L(f-f(0))+f(0)L(1)\\
	&=L(\mathcal{I}(F))+f(0)L(1)
	=\langle F,G\rangle_{A^2_{\frac1p-1}}+f(0)L(1)\\
	&=\langle f',G\rangle_{A^2_{\frac1p-1}}+f(0)L(1),\quad f\in D^p_{p-1}.
	\end{split}
	\end{equation*}
By denoting
	$$
	w_{n,p}=\int_0^1r^{2n+1}(1-r^2)^{\frac1p-1}\,dr,\quad n\in\N\cup\{0\},
	$$
we deduce
	\begin{equation*}
	\begin{split}
	\langle f',G\rangle_{A^2_{\frac1p-1}}
	&=\lim_{r\to 1^-}\int_\D rf'(rz)\overline{G(z)}(1-|z|^2)^{\frac1p-1}\,dA(z)\\
	&= \lim_{r\to 1^-}\int_\D \left(\sum_{n=0}^\infty\widehat{f}(n+1)(n+1)r(rz)^n \right) \left(\sum_{k=0}^\infty\overline{\widehat{G}(k)}\overline{z}^k
    \right)(1-|z|^2)^{\frac1p-1}\,dA(z)\\
	&= \lim_{r\to 1^-}2\pi\sum_{n=0}^\infty\widehat{f}(n+1)r^{n+1}\left(\frac{w_{n,p}}{w_{n+1,p}}(n+1)\overline{\widehat{G}(n)}\right)w_{n+1,p}\\
	&=\lim_{r\to 1^-}\int_\D \sum_{n=0}^\infty\widehat{f}(n+1)(rz)^{n+1}
	\sum_{k=0}^\infty\left(\frac{w_{k,p}}{w_{k+1,p}}(k+1)\overline{\widehat{G}(k)}\right)\overline{z}^{k+1}(1-|z|^2)^{\frac1p-1}\,dA(z)\\
	&=\langle f-f(0),K_p\rangle_{A^2_{\frac1p-1}},
	\end{split}
	\end{equation*}
where
	$$
	K_p(z)=\sum_{k=0}^\infty\frac{w_{k,p}}{w_{k+1,p}}(k+1)\widehat{G}(k) z^{k+1},\quad z\in\D.
	$$
In the case $p=1$ we have
	$$
	K_1(z)=\sum_{k=0}^\infty(k+2)\widehat{G}(k) z^{k+1}
	=\frac{d}{dz}\left(z^2G(z)\right),\quad z\in\D,
	$$
and hence $K_1\in\B^2$ by \eqref{Eq-alpha-bloch}. To obtain the same conclusion for each $0<p<1$ we first observe that 
$J(z)=\sum_{k=0}^\infty(k+1)\widehat{G}(k)z^{k+1}=\frac{d}{dz}\left(zG(z)\right)$, and thus $J\in\B^2$. Therefore it suffices to show that 
$\lambda_p(z)=\sum_{n=0}^\infty\frac{w_{n,p}}{w_{n+1,p}}z^{n+1}$ is a coefficient multiplier of $\B^2$ for each $0<p\le 1$.

To see this, for each $\beta\in\mathbb{N}$, denote $D^\beta f(z)=\sum_{n=0}^\infty (n+1)^\beta\widehat{f}(n)z^n$ for all
 $f\in\H(\D)$, and for simplicity write $Df$ instead of $D^1f$. We claim that
\begin{equation}\label{eq:M_1}
M_1(r,D\lambda_p)\lesssim\frac1{1-r},\quad 0\le r<1,
\end{equation}
the proof of which is postponed for a moment. Direct calculations show that
	$$
	\|f\|_{\B^2}\asymp\sup_{z\in\D}|Df(z)|(1-|z|)^2
	\asymp\sup_{z\in\D}|D^2f(z)|(1-|z|)^3,\quad f\in\H(\D),
	$$
and hence \eqref{eq:M_1} yields
	\begin{equation*}
	\begin{split}
	|D^2(f*\lambda_p)(r^2e^{it})|
	&=|(Df*D\lambda_p)(r^2e^{it})|
	=\left|\frac1{2\pi}\int_0^{2\pi}Df(re^{i(t+\theta)})D\lambda_p(re^{-i\theta})\,d\theta\right|\\
	&\le M_\infty(r,Df)M_1(r,D\lambda_p)\lesssim\frac{M_1(r,D\lambda_p)}{(1-r)^2}\lesssim\frac{1}{(1-r)^3},\quad f\in\B^2.
	\end{split}
	\end{equation*}	
It follows that $f*\lambda_p\in\B^2$ for all $f\in\B^2$ and $0<p\le1$. Thus $K_p\in\B^2$ for each $0<p\le1$. By choosing $H_p=K_p+\frac{\overline{L(1)}}{\omega_{0,p}}\in\B^2$, we deduce $L(f)=\langle f,H_p\rangle_{A^2_{\frac1p-1}}$ for all $f\in D^p_{p-1}$. 

To complete the proof, it remains to establish~\eqref{eq:M_1}. To do this, we will use the families of polynomials defined by \eqref{eq:polynomials} and \eqref{vn}. It follows from \eqref{propervn} that
    \begin{equation}\label{fbnorm}
    M_1(r,D\lambda_p)=\|(D\lambda_p)_r\|_{H^1}\le C(p)+\sum_{n=2}^\infty\|V_n\ast(D\lambda_p)_r\|_{H^1},
    \end{equation}
where $(D\lambda_p)_r(z)=\sum_{n=1}^\infty n\frac{w_{n-1,p}}{w_{n,p}}r^{n}z^{n}$. Next, for each $n\in\N\setminus\{1\}$ and
$r\in\left[\frac12,1\right)$, consider
    $$
    F_n(x)=x\frac{w_{x-1,p}}{w_{x,p}}r^{x}\chi_{[2^{n-1},2^{n+1}]}(x),\quad x\in\RR.
    $$
Since for each radial weight $\nu$ there exists a constant $C=C(\nu)>0$ such that
    \begin{equation*}
    \int_0^1 s^x\left(\log\frac1s\right)^n\nu(s)\,ds\le C\int_0^1s^x\nu(s)\,ds, \quad n\in\{1,2\},\quad x\ge 2,
    \end{equation*}
it follows by a direct calculation that
    $$
    |F_n''(x)|\le C|F_n(x)|, \quad n\in\N\setminus\{1\}, \quad r\in\left[\frac12,1\right),\quad x\ge2,
    $$
for some constant $C=C(\om)>0$. Therefore,
    \begin{equation*}
    \begin{split}
    A_{F_n,2}&=\max_{x\in[2^{n-1},2^{n+1}]}|F_n(x)|+\max_{x\in[2^{n-1},2^{n+1}]}|F_n''(x)|
		\lesssim \max_{x\in[2^{n-1},2^{n+1}]}|F_n(x)|\\
    &
    \lesssim \max_{x\in[2^{n-1},2^{n+1}]} (x+1) r^{x+1}
    \lesssim 2^{n} r^{2^{n-1}},\quad n\in\N\setminus\{1\}.
    \end{split}
    \end{equation*}
For each $n\in\N\setminus\{1\}$, choose a $C^\infty$-function $\Phi_n$ with compact support contained in $[2^{n-2},2^{n+2}]$ such that
$\Phi_n=F_n$ on $[2^{n-1},2^{n+1}]$ and
    \begin{equation}\label{phin}
    A_{\Phi_n,2}=\max_{x\in\mathbb{R}}|\Phi_n(x)|+\max_{x\in\mathbb{R}}|\Phi''_n(x)|
    \lesssim 2^{n} r^{2^{n-1}},\quad n\in\N\setminus\{1\}.
    \end{equation}
Since
    $$
    W_1^{\Phi_n}(z)
    =\sum_{k\in\mathbb{Z}}\Phi_n\left(k\right)z^k
    =\sum_{k\in\mathbb{Z}\cap[2^{n-2},2^{n+2}]}\Phi_n\left(k\right)z^k,
    $$
the identity \eqref{vn} yields
    \begin{equation*}
    \begin{split}
    \left(V_n\ast(D\lambda_p)_r\right)(z)
    &=\sum_{k=2^{n-1}}^{2^{n+1}-1}\psi\left(\frac{k}{2^{n-1}}\right)k\frac{w_{k-1,p}}{w_{k,p}}r^{k}z^k
    =\sum_{k=2^{n-1}}^{2^{n+1}-1}\psi\left(\frac{k}{2^{n-1}}\right)\Phi_n(k)z^k\\
    &=\left(W_1^{\Phi_n}\ast V_n\right)(z),\quad n\in\N\setminus\{1\}.
    \end{split}
    \end{equation*}
This together with Theorem~\ref{th:cesaro}, \eqref{phin} and \eqref{propervn} implies
    \begin{equation*}
    \begin{split}
    \|V_n\ast(D\lambda)_r\|_{H^1}
    &=\|W_1^{\Phi_n}\ast V_n\|_{H^1}
    \lesssim A_{\Phi_n,2}\|V_n\|_{H^1}
    \lesssim  2^{n} r^{2^{n-1}}\|V_n\|_{H^1}\\
    &\lesssim 2^{n} r^{2^{n-1}},\quad r\in \left[\frac12,1\right),\quad n\in\N\setminus\{1\},
    \end{split}
    \end{equation*}
which combined with \eqref{fbnorm} gives
    \begin{equation}\label{endstep1new1}
    M_1(r,D\lambda_p)
    \lesssim\sum_{n=2}^\infty2^{ n}r^{2^{n-1}}
    \lesssim\frac{1}{(1-r)}, \quad  r\in\left[\frac12,1\right).
    \end{equation}
This implies \eqref{eq:M_1}, and finishes the proof.
\end{proof}

For two Banach spaces $X,Y\subset\H(\D)$, let $[X,Y]=\{g\in\H(\D): f*g\in Y \textrm{ for every }f\in X\}$ denote the space of coefficient multipliers from $X$ to $Y$. The dual space of the Banach space $\HL_p$, with $1<p<\infty$, with respect to the $H^2$-pairing is given in the next lemma. This result is needed in the proof of Theorem~\ref{Theorem:Hardy-p<2}.

\begin{lemma}\label{le:dualHLp}
Let $1<p<\infty$. Then $(\HL_p)^\star\simeq\HL_{p'}$ via the $H^2$-pairing with equivalence of norms.
\end{lemma}

\begin{proof}
Let $\mathcal{A}$ denote the disc algebra. By \cite[Proposition~1.3]{P}, the dual of $\HL_p$ can be identified with $[\HL_p, \mathcal{A}]$ under the $H^2$-pairing. We claim that $[\HL_p,\mathcal{A}]=\HL_{p'}$ for each $1<p<\infty$. 

To see this, first observe that for all $g\in\HL_{p'}$ and $f\in\HL_p$, H\"{o}lder's inequality yields
	$$
	\sup_{z\in\overline{\D}}|(f*g)(z)|
	\le\sum_{n=0}^{\infty}|\widehat{f}(n)||\widehat{g}(n)|
	\le\|f\|_{\HL_p}\|g\|_{\HL_{p'}}<\infty,
	$$
which shows that $\HL_{p'}\subset[\HL_p, \mathcal{A}]$.

Conversely, if $g(z)=\sum_{n=0}^{\infty}\widehat{g}(n)z^n\in[\HL_p,\mathcal{A}]$, then
	$$\infty>\sup_{z\in\overline{\D}}\left|\sum_{n=0}^{\infty}\widehat{f}(n)\widehat{g}(n)z^n \right|
	\ge
	\left|\sum_{n=0}^{\infty}\widehat{f}(n)(n+1)^{\frac{p-2}{p}}\widehat{g}(n)(n+1)^{\frac{p'-2}{p'}}\right|,\quad f\in\HL_p.
	$$
Since $f\in\HL_p$, that is, $\{\widehat{f}(n)(n+1)^{\frac{p-2}{p}}\}\in\ell^p$, the duality $(\ell^p)^\star\simeq\ell^{p'}$ yields $\{\widehat{g}(n)(n+1)^{\frac{p'-2}{p'}}\}\in\ell^{p'}$. Thus $g\in\HL_{p'}$, and the proof is finished.
\end{proof}

Recall that the space $\HL_{\infty}$ consists of $f\in\H(\D)$ such that 
	$$
	\|f\|_{\HL_{\infty}}=\sup_{n\in\N\cup\{0\}}\left(|\widehat{f}(n)|(n+1)\right)<\infty.	
	$$
The last lemma of the section describes $(\HL_1)^\star$. It will be used in the proof of Theorem~\ref{Theorem:p=1}.

\begin{lemma}\label{le:dualHL1}
$(\HL_1)^\star\simeq\HL_{\infty}$ via the $H^2$-pairing with equivalence of norms.
\end{lemma}

\begin{proof}
By \cite[Proposition~1.3]{P}, the dual of $\HL_1$ can be identified with $[\HL_1,\mathcal{A}]$ under the $H^2$-pairing. We claim that $[\HL_1,\mathcal{A}]=\HL_{\infty}$.

On one hand, 
	$$
	\sup_{z\in\overline{\D}}|(f*g)(z)|
	\le\sum_{n=0}^{\infty}\frac{|\widehat{f}(n)|}{n+1}\left(|\widehat{g}(n)|(n+1)\right)
	\le\|f\|_{\HL_1}\|g\|_{\HL_{\infty}},
	$$
and hence $\HL_{\infty}\subset[\HL_1,\mathcal{A}]$.

On the other hand, let $g(z)=\sum_{n=0}^{\infty}\widehat{g}(n)z^n$ be a member of $[\HL_1,\mathcal{A}]$. Then
	$$
	\infty>\sup_{z\in\overline{\D}}\left|\sum_{n=0}^{\infty}\widehat{f}(n)\widehat{g}(n)z^n\right|
	\ge\left|\sum_{n=0}^{\infty}\frac{\widehat{f}(n)}{n+1}\left(\widehat{g}(n)(n+1)\right)\right|,\quad f\in\HL_1.
	$$
Since $f\in\HL_1$ if and only if $\{\frac{\widehat{f}(n)}{n+1}\}_{n=0}^\infty\in\ell^1$, the duality $(\ell^1)^\star\simeq\ell^{\infty}$ yields $g\in\HL_{\infty}$. Thus $\HL_{\infty}=[\HL_1,\mathcal{A}]$.
\end{proof}

\section{Hardy, Hardy-Littlewood and Dirichlet-type spaces with $1<p<\infty$}\label{Sec:pili}

The main  aim of this section is to prove Theorem~\ref{Theorem:Hardy-p<2}. To do that some notation and auxiliary results are needed. For each $g\in\H(\D)$, with Maclaurin series expansion $g(z)=\sum_{k=0}^\infty \widehat{g}(k)z^k$, consider the dyadic polynomials defined by $\Delta_0 g (z)=g(0)$ and $\Delta_n g (z)=\sum_{k=2^{n}}^{2^{n+1}-1}\widehat{g}(k)z^k$ for all $n\in\N$ and $z\in\D$. Then, obviously, $g=\sum_{n=0}^\infty\Delta_n g$. Further, write $\Delta_0=1$ and $\Delta_n(z)= \sum_{k=2^{n}}^{2^{n+1}-1}z^k$ for all $n\in\N$ and $z\in\D$. Then \cite[Lemma~2.7]{CPPR} shows that
	\begin{equation}\label{Deltanorm}
	\|\Delta_n\|_{H^p}\asymp 2^{\frac{n}{p'}},\quad 1<p<\infty,\quad n\in \N\cup\{0\}.
	\end{equation}
For $a\in\D$, denote $f_a(z)=f(az)$ for all $z\in\D$. With these preparations we can state the first auxiliary result.

\begin{proposition}\label{pr:j1}
Let $1<q<\infty$ and $g\in \H(\D)$ such that $\sum_{k=0}^\infty |\widehat{g}(k)|<\infty$. Then there exists a constant $C=C(q)>0$ such that
	\begin{equation}
	\begin{split}\label{eq:j1}
	\left\| \Delta_j f\ast G^{H^2}_{g,z} \right\|_{H^q}
	&=\left\|\sum_{k=2^j}^{2^{j+1}-1}\widehat{f}(k)\sum_{n=0}^\infty\frac{(n+1)\overline{\widehat{g}(n+1)}\overline{z}^{n+k+1} }{n+k+1}\z^k	\right\|_{H^q}\\
	&\le C\left\| \Delta_j f_{\overline{z}}  \right\|_{H^q}  \sum_{n=0}^\infty \frac{(n+1) |\widehat{g}(n+1)| |z|^{n+1} }{n+ 2^{j-1}+1},\quad z\in\D,
	\end{split}
	\end{equation}
for all $f\in\H(\D)$ and $j\in\N$.
\end{proposition}

\begin{proof}
For each $j\in\N$ and $z\in\D$, let us consider the $C^\infty$-function
	$$
	\Psi_{2^j,z}(s)=\sum_{n=0}^\infty \frac{(n+1)\overline{\widehat{g}(n+1)}\overline{z}^{n+1} }{n+1+2^{j}s},\quad s>0.
	$$
Then
\begin{equation}\label{eq:j2}
|\Psi_{2^j,z}(s)|\le \sum_{n=0}^\infty \frac{(n+1)|\widehat{g}(n+1)z^{n+1}| }{n+1+2^{j-1}},\quad s\ge \frac{1}{2}.
\end{equation}
Further,
	$$
	(\Psi_{2^j,z})'(s)=-2^j\sum_{n=0}^\infty \frac{(n+1)\overline{\widehat{g}(n+1)}\bar{z}^{n+1} }{(n+1+2^{j}s)^2},\quad s>0,$$
and hence
	\begin{equation}
	\begin{split}\label{eq:j3}
	|(\Psi_{2^j,z})'(s)|
	&\le2^j\sum_{n=0}^\infty\frac{(n+1)|\widehat{g}(n+1)z^{n+1}|}{(n+1+2^{j-1})^2}\\
	&\le2\sum_{n=0}^\infty\frac{(n+1)|\widehat{g}(n+1)z^{n+1}|}{n+1+2^{j-1}},\quad s\ge\frac{1}{2}.
	\end{split}
	\end{equation}
Therefore, by using \eqref{eq:j2} and \eqref{eq:j3}, we can find a $C^\infty$-function $\Phi_{2^j,z}$ and an absolute constant $C>0$ such that $\supp\Phi_{2^j,z}\subset\left( \frac{1}{2},4\right)$, $\Phi_{2^j,z}(s)=\Psi_{2^j,z}(s)$ for all $s\in [1,2]$ and
	$$
  A_{\Phi_{2^j,z},1}
	=\max_{s\in\mathbb{R}}|\Phi_{2^j,s}(s)|
	+\max_{s\in\mathbb{R}}|\Phi_{2^j,z}'(s)|
	\le C\sum_{n=0}^\infty \frac{(n+1)|\widehat{g}(n+1)z^{n+1}|}{n+1+2^{j-1}}.
  $$
Hence
	\begin{equation*}
	\begin{split}
	\Delta_j f\ast G^{H^2}_{g,z}(\z)
	&=\sum_{k=2^j}^{2^{j+1}-1}\widehat{f}(k)\overline{z}^{k}
	\sum_{n=0}^\infty \frac{(n+1)\overline{\widehat{g}(n+1)}\bar{z}^{n+1} }{n+k+1} \z^k\\
	&=\sum_{k=2^j}^{2^{j+1}-1}\widehat{f}(k)\overline{z}^{k}\Phi_{2^j,z}\left(\frac{k}{2^j}\right)\z^k
	=\left(\Delta_jf_{\overline{z}}\ast W^{\Phi_{2^j,z}}_{2^j}\right)(\z),\quad j\in\N.
	\end{split}
	\end{equation*}
Using now Theorem~\ref{th:cesaro} we find a constant $C=C(q)>0$ such that
	\begin{equation*}
	\begin{split}
	\left\|\Delta_j f\ast G^{H^2}_{g,z}(\z) \right\|_{H^q}
	&=\left\| \Delta_j f_{\bar{z}}\ast W^{\Phi_{2^j,z}}_{2^j}\right\|_{H^q}
	\le CA_{\Phi_{2^j,z},1}\left\|\Delta_j f_{\overline{z}}\right\|_{H^q}\\
	&\le C\left\| \Delta_j f_{\overline{z}}\right\|_{H^q} \sum_{n=0}^\infty \frac{(n+1)|\widehat{g}(n+1)z^{n+1}|}{n+1+2^{j-1}}.
	\end{split}
	\end{equation*}
This finishes the proof.
\end{proof}

The next result gives a sufficient condition for $T_g:~H^p\rightarrow H^\infty$ to be bounded and establishes the operator norm estimate \eqref{poi} announced in the introduction.

\begin{theorem}\label{th:dpsuf}
Let $1<p<\infty$ and $g\in\H(\D)$ such that $$\sum_{k=0}^{\infty}(k+1)^{p'-2}\left(\sum_{n=0}^{\infty}\frac{(n+1)|\widehat{g}(n+1)|}{n+k+1}\right)^{p'}<\infty.$$ 
If $X_p\in\{H^p,D^p_{p-1},\HL_p\}$,  then $T_g:~X_p\rightarrow H^\infty$ is bounded and
	$$
	\|T_g\|_{ X_p \rightarrow H^\infty}^{p'}
  \lesssim
  \sum_{k=0}^{\infty}(k+1)^{p'-2} \left(\sum_{n=0}^{\infty}\frac{(n+1)|\widehat{g}(n+1)|}{n+k+1}\right)^{p'}.
	$$
\end{theorem}

\begin{proof}
We begin with the case $X_p=D^p_{p-1}$. By Lemma~\ref{le:dualdp} and \cite[Theorem~2.2]{CPPR}, $T_g:~D^p_{p-1}\rightarrow H^\infty$ is bounded if and only if $\sup_{z\in\D}\|G^{H^2}_{g,z}\|_{ D^{p'}_{p'-1}}<\infty$, and moreover,
	\begin{equation}\label{eq:j4}
	\|T_g\|_{D^p_{p-1}\rightarrow H^\infty}^{p'}
	\asymp\sup_{z\in\D}\|G^{H^2}_{g,z}\|^{p'}_{D^{p'}_{p'-1}}.
  \end{equation}
Further, for each $1<q<\infty$, \cite[Theorem~2.1]{MatPav} yields
	\begin{equation}\label{eq:j5}
	\|F\|^q_{D^q_{q-1}}\asymp\sum_{j=0}^\infty\| \Delta_j\ast F\|^q_{H^q},\quad F\in\H(\D).
	\end{equation}
Therefore, by combining \eqref{eq:j4}, \eqref{eq:j5}, Proposition~\ref{pr:j1} and \eqref{Deltanorm}, we deduce
	\begin{equation}
	\begin{split}\label{eq:TgDp}
	\|T_g\|_{D^p_{p-1}\rightarrow H^\infty}^{p'}
	&\asymp\sup_{z\in\D}\|G^{H^2}_{g,z}\|^{p'}_{D^{p'}_{p'-1}}
	\asymp\sup_{z\in\D}\sum_{j=0}^\infty \|\Delta_j \ast  G^{H^2}_{g,z}\|_{H^{p'}}^{p'}\\
	&\lesssim\sup_{z\in\D}\sum_{j=0}^\infty
	\|\Delta_j\|_{H^{p'}}^{p'}\left(\sum_{n=0}^\infty\frac{(n+1)|\widehat{g}(n+1)| |z|^{n+1} }{n+ 2^{j-1}+1}\right)^{p'}\\
	&\lesssim\sum_{j=0}^\infty 2^{j(p'-1)} \left( \sum_{n=0}^\infty \frac{(n+1) |\widehat{g}(n+1)|  }{n+ 2^{j-1}+1}\right)^{p'}
	\\
	&\lesssim \sum_{k=0}^{\infty}(k+1)^{p'-2} \left( \sum_{n=0}^\infty \frac{(n+1) |\widehat{g}(n+1)|  }{n+ k+1}  \right)^{p'}.
	\end{split}
	\end{equation}
Thus the assertion is proved for $X_p=D^p_{p-1}$.

Next we deal with the case $X_p=\HL_p$. By Lemma~\ref{le:dualHLp} and \cite[Theorem~1.1]{CPPR}, $T_g:\HL_p\rightarrow H^\infty$ is bounded if and only if $\sup_{z\in\D}\|G^{H^2}_{g,z}\|_{\HL_{p'}}<\infty$, and moreover,
	\begin{equation}\label{eq:j4+}
	\|T_g\|_{\HL_p\rightarrow H^\infty}^{p'}
	\asymp\sup_{z\in\D}\|G^{H^2}_{g,z}\|^{p'}_{\HL_{p'}}.
  \end{equation}
But
	\begin{equation*}
	\begin{split}
	\|G^{H^2}_{g,z}\|^{p'}_{\HL_{p'}}
	&=\sum_{k=0}^{\infty}(k+1)^{p'-2}\left|\sum_{n=0}^\infty\frac{(n+1)\overline{\widehat{g}(n+1)}\bar{z}^{n+k+1} }{n+k+1}\right|^{p'}\\
  &\le \sum_{k=0}^{\infty}(k+1)^{p'-2} \left( \sum_{n=0}^\infty \frac{(n+1) |\widehat{g}(n+1)|  }{n+ k+1}  \right)^{p'},
	\end{split}
	\end{equation*}
and thus $T_g:\HL_p\rightarrow H^\infty$ is bounded and
	$$
	\|T_g\|_{\HL_{p}\rightarrow H^\infty}^{p'}\le\sum_{k=0}^{\infty}
   (k+1)^{p'-2} \left( \sum_{n=0}^\infty \frac{(n+1) |\widehat{g}(n+1)|  }{n+ k+1}  \right)^{p'}.
	$$
Bearing in mind \eqref{Hardy1} and \eqref{Hardy2}, the remaining case $X_p=H^p$ follows from \cite[Theorem~1.1]{CPPR}, the well-known identification $(H^p)^\star\simeq H^{p'}$ via the $H^2$-pairing and the two cases already proven.
\end{proof}

Despite the inclusions in \eqref{Hardy1} and \eqref{Hardy2} are strict unless $p=2$, if one restricts to the class of power series with non-negative decreasing coefficients, then the following statements hold by \cite{HD}, \cite{P2} and \cite[Chapter~XII, Lemma~6.6]{Z}.

\begin{letterlemma}\label{Eq1}
Let $1\le p<\infty$, then there exist constants $C_1=C_1(p)>0$, $C_2=C_2(p)>0$ and $C_3=C_3(p)>0$ such that
	$$ 
	\|f\|^p_{H^p}
	\le C_1\|f\|^p_{D^p_{p-1}}
	\le C_2 \|f\|^p_{\HL_p}
	\le C_3 \|f\|^p_{H^p},
	$$
for all $f\in\H(\D)$ such that its Maclaurin coefficients $\{\widehat{f}(n)\}_{n=0}^\infty$ form a sequence of non-negative numbers decreasing to zero. In particular, 
	$$
	f\in H^p\quad\Longleftrightarrow\quad 
	f\in D^p_{p-1}\quad\Longleftrightarrow\quad 
	f\in\HL_p
	$$
for every such $f$.
\end{letterlemma}

We are now ready to prove Theorem~\ref{Theorem:Hardy-p<2}.

\medskip

\noindent{\emph {Proof of}} Theorem~\ref{Theorem:Hardy-p<2}. Assume first that $T_g:~X_p\rightarrow H^\infty$ is bounded. Then $g\in H^\infty$, and hence
	\begin{equation}\label{eq:j10}
	\sum_{n=0}^\infty |\widehat{g}(n)|=\sum_{n=0}^\infty\widehat{g}(n)<\infty.
	\end{equation}
Lemmas~\ref{le:dualdp} and~\ref{le:dualHLp} together with the well-known identification of $(H^p)^\star$ as $H^{p'}$ via the $H^2$-pairing imply $(X_p)^\star\simeq X_{p'}$. Therefore \cite[Theorem~1.1]{CPPR} yields
	\begin{equation}\label{q1}
	\begin{split}
	\|T_g\|_{X_p\rightarrow H^\infty}^{p'}
	&\asymp\sup_{z\in\D}\|G^{H^2}_{g,z}\|_{X_{p'}}^{p'}
     \gtrsim \sup_{x\in (0,1)}\|G^{H^2}_{g,x}\|_{X_{p'}}^{p'}.
	\end{split}
	\end{equation}
Since $G^{H^2}_{g,x}(\z)=\sum_{k=0}^\infty \left(\sum_{n=0}^\infty\frac{(n+1)\widehat{g}(n+1)x^{n+k+1}}{n+k+1} \right)\z^k$, for each $x\in(0,1)$, the Maclaurin coefficients
	$$
	\widehat{G^{H^2}_{g,x}}(k)=\sum_{n=0}^\infty \frac{(n+1)\widehat{g}(n+1)x^{n+k+1}}{n+k+1}, \quad k\in\N\cup\{0\},
	$$
form a sequence of non-negative and decreasing numbers. Therefore \eqref{q1}, Lemma~\ref{Eq1} and \eqref{eq:j10} imply
	\begin{equation*}
	\begin{split}
	\|T_g\|_{X_p\rightarrow H^\infty}^{p'}&\gtrsim
    \sup_{x\in (0,1)}\|G^{H^2}_{g,x}\|_{\HL_{p'}}^{p'}
	\asymp\sup_{x\in(0,1)}\sum_{k=0}^{\infty}(k+1)^{p'-2}\left(\sum_{n=0}^{\infty}\frac{(n+1)\widehat{g}(n+1)}{n+k+1}x^{n+k+1}\right)^{p'}\\
	&\asymp\sum_{k=0}^{\infty}(k+1)^{p'-2}\left(\sum_{n=0}^{\infty}\frac{(n+1)\widehat{g}(n+1)}{n+k+1}\right)^{p'}.
	\end{split}
	\end{equation*}
Thus \eqref{pipeli} holds.

Conversely, if \eqref{pipeli} is satisfied, then $T_g:~X_p\rightarrow H^\infty$ is bounded and
	$$
	\|T_g\|_{X_p\rightarrow H^\infty}^{p'}
	\lesssim \sum_{k=0}^{\infty}(k+1)^{p'-2}\left(\sum_{n=0}^{\infty}\frac{(n+1)\widehat{g}(n+1)}{n+k+1}\right)^{p'}
	$$
by Theorem~\ref{th:dpsuf}. The norm estimate \eqref{eq:j11} follows from the above inequalities.

To complete the proof we still need to show that $T_g:~X_p\rightarrow H^\infty$ is in fact compact if \eqref{pipeli} is satisfied. To see this, let first $X_p=\HL_p$. Further, let $\{f_n\}$ such that $\sup_{n}\|f_n\|_{\HL_p}<\infty$
and $f_n\to0$ uniformly on compact subsets of $\D$ as $n\to\infty$.
For each $\e>0$ there exists $k_0=k_0(\e)\in\N$ such that
	\begin{equation*}\label{eq:j16}
	\sum_{k=k_0}^{\infty}(k+1)^{p'-2}\left(\sum_{n=0}^{\infty}\frac{(n+1)\widehat{g}(n+1)}{n+k+1}\right)^{p'}<\e^{p'}.
	\end{equation*}
Moreover, by the uniform convergence we may pick up an $n_0=n_0(\e)\in\N$ such that
	\begin{equation*}
	\sup_{n\ge n_0}\sum_{k=0}^{k_0-1} (k+1)^{p-2}|\widehat{f_n}(k)|^p<\e^{p}.
	\end{equation*}
Then, H\"older's inequality yields
	\begin{equation*}
	\begin{split}
	\|T_g(f_n)\|_{H^\infty}
	&=\sup_{z\in\D}|\langle f_n,G^{H^2}_{g,z}\rangle_{H^2}|
	\le\sum_{k=0}^\infty |\widehat{f_n}(k)|\left(\sum_{n=0}^{\infty}\frac{(n+1)\widehat{g}(n+1)}{n+k+1}\right)\\
	&=\sum_{k=0}^{k_0-1} |\widehat{f_n}(k)|\left(\sum_{n=0}^{\infty}\frac{(n+1)\widehat{g}(n+1)}{n+k+1}\right)
	+\sum_{k=k_0}^\infty |\widehat{f_n}(k)|\left(\sum_{n=0}^{\infty}\frac{(n+1)\widehat{g}(n+1)}{n+k+1}\right)\\
	&\le\left(\sum_{k=0}^{k_0-1}(k+1)^{p-2} |\widehat{f_n}(k)|^p \right)^{\frac1p}
 \left(\sum_{k=0}^{k_0-1}(k+1)^{p'-2}\left(\sum_{n=0}^{\infty}\frac{(n+1)\widehat{g}(n+1)}{n+k+1}\right)^{p'} \right)^{\frac1{p'}}\\
	&\quad+
 \left(\sum_{k=k_0}^{\infty}(k+1)^{p-2} |\widehat{f_n}(k)|^p  \right)^{\frac1p} \left(\sum_{k=k_0}^{\infty}(k+1)^{p'-2}\left(\sum_{n=0}^{\infty}\frac{(n+1)\widehat{g}(n+1)}{n+k+1}\right)^{p'} \right)^{\frac1{p'}}\\
	&\le\e\left(\left(\sum_{k=0}^{\infty}(k+1)^{p'-2}
	\left(\sum_{n=0}^{\infty}\frac{(n+1)\widehat{g}(n+1)}{n+k+1}\right)^{p'}\right)^{1/p'}+ \sup_{n}\|f_n\|_{\HL_p}\right)\\
	&\lesssim \e,\quad n\ge n_0,
	\end{split}
	\end{equation*}
and hence $\lim_{n\to \infty}\|T_g(f_n)\|_{H^\infty}=0$. Therefore $T_g:\HL_p\to H^\infty$ is compact by \cite[Lemma~3.6]{Tjani}.

Let now $X_p=D^p_{p-1}$. We first show that
	\begin{equation}\label{eq:j16}
	\lim_{R\rightarrow1^-}\sup_{z\in\D}\int_{\D\backslash\overline{\D(0,R)}}|(G^{H^2}_{g,z})'(w)|^{p'}(1-|w|)^{p'-1}\,dA(w)=0,
	\end{equation}
and then we use this fact to prove the compactness of $T_g:D^p_{p-1} \to H^\infty$.

If $2<p'<\infty$, then \eqref{Hardy2} and Fubini's theorem yield
	\begin{equation*}
	\begin{split}
	&\lim_{R\rightarrow1^-}\sup_{z\in\D}\int_{\D\backslash\overline{\D(0,R)}}|(G^{H^2}_{g,z})'(w)|^{p'}(1-|w|)^{p'-1}dA(w)\\
	&=\lim_{R\rightarrow1^-}\sup_{z\in\D}\int_R^1\int_0^{2\pi}\left|\sum_{k=0}^{\infty}(k+1)
	\sum_{n=0}^{\infty}\frac{(n+1)\widehat{g}(n+1)z^{n+k+2}}{n+k+2}r^ke^{ik\theta}\right|^{p'} d\theta (1-r)^{p'-1}r\,dr\\
	&=2\pi\lim_{R\rightarrow1^-}\sup_{z\in\D}\int_R^1 \| (G^{H^2}_{g,z})'_r\|^{p'}_{H^{p'}} (1-r)^{p'-1}r\,dr\\
	&\lesssim\lim_{R\rightarrow1^-}\int_R^1\sum_{k=0}^{\infty}(k+1)^{p'-2}\left(\sum_{n=0}^\infty\frac{(n+1)\widehat{g}(n+1)}{n+k+1}\right)^{p'}
	(k+1)^{p'}r^{kp'+1}(1-r)^{p'-1}dr\\
	&=\lim_{R\rightarrow1^-}\sum_{k=0}^{\infty}(k+1)^{p'-2}\left(\sum_{n=0}^\infty\frac{(n+1)\widehat{g}(n+1)}{n+k+1}\right)^{p'}
	(k+1)^{p'}\int_R^1r^{kp'+1}(1-r)^{p'-1}dr,
	\end{split}
	\end{equation*}
where
	$$
	(k+1)^{p'}\int_R^1r^{kp'+1}(1-r)^{p'-1}dr\leq(k+1)^{p'}\int_0^1r^{kp'+1}(1-r)^{p'-1}dr \asymp 1,\quad k\in\N\cup\{0\}.
	$$
The dominated convergence theorem now implies \eqref{eq:j16}.

If $1<p'\le2$, then Proposition~\ref{pr:j1} and an argument similar to that used in the proof of \eqref{eq:TgDp} allows us to find a constant $C=C(p)>0$ such that
	$$
	\| (G^{H^2}_{g,z})'_r\|^{p'}_{D^{p'}_{p'-1}}\le C\sum_{k=0}^{\infty}(k+1)^{p'-2}\left(\sum_{n=0}^\infty\frac{(n+1)\widehat{g}(n+1)}{n+k+1}\right)^{p'}
	(k+1)^{p'}r^{kp'},\quad 0\le r<1.
	$$
This together with \eqref{Hardy1} implies
  \begin{equation*}
	\begin{split}
	&\lim_{R\rightarrow1^-}\sup_{z\in\D}\int_{\D\backslash\overline{\D(0,R)}}|(G^{H^2}_{g,z})'(w)|^{p'}(1-|w|)^{p'-1}dA(w)\\
     & =2\pi\lim_{R\rightarrow1^-}\sup_{z\in\D}\int_R^1 \| (G^{H^2}_{g,z})'_r\|^{p'}_{H^{p'}} (1-r)^{p'-1}r\,dr
     \\ & \lesssim  \lim_{R\rightarrow1^-}\sup_{z\in\D}\int_R^1\| (G^{H^2}_{g,z})'_r\|^{p'}_{D^{p'}_{p'-1}}(1-r)^{p'-1}r\,dr
    \\ &\lesssim\lim_{R\rightarrow1^-}\int_R^1\sum_{k=0}^{\infty}(k+1)^{p'-2}\left(\sum_{n=0}^\infty\frac{(n+1)\widehat{g}(n+1)}{n+k+1}\right)^{p'}
	(k+1)^{p'}r^{kp'+1}(1-r)^{p'-1}dr=0.
	\end{split}
	\end{equation*}
Consequently, \eqref{eq:j16} holds for each $1<p<\infty$.

Let now $\{f_n\}$ such that $\sup_{n}\|f_n\|_{D^p_{p-1}}<\infty$
and $f_n\to0$ uniformly on compact subsets of $\D$. By \eqref{eq:j16}, for each $\varepsilon>0$, there exists $R=R(\varepsilon)\in(0,1)$  such that
	$$
	\sup_{z\in\D}\int_{\D\backslash\overline{\D(0,R)}}|(G^{H^2}_{g,z})'(w)|^{p'}(1-|w|)^{p'-1}dA(w)<\varepsilon^{p'}.
	$$
Further, by the uniform convergence we may choose $N=N(\varepsilon,R)\in\N$ such that $\max\{|f_n(0)|,|f_n'(\xi)|\}<\varepsilon$ for all $n\geq N$ and $\xi\in\overline\D(0,R)$. Therefore \cite[(2.4) and (4.4)]{CPPR} and H\"{o}lder's inequality yield
	\begin{equation*}
	\begin{split}
	\|T_g(f_n)\|_{H^\infty}
	&=\sup_{z\in\D}|\langle f_n,G^{H^2}_{g,z}\rangle_{H^2}|\\
	&\lesssim\sup_{z\in\D}\left|\int_{\D}f_n'(w)\overline{(G^{H^2}_{g,z})'(w)}\log\frac1{|w|}\,dA(w)\right|+|f_n(0)| \|g\|_{H^\infty}\\
	&\lesssim\sup_{z\in\D}\left(\left(\int_{\overline{\D(0,R)}}+\int_{\D\backslash\overline{\D(0,R)}}\right)
	|f_n'(w)||(G^{H^2}_{g,z})'(w)|(1-|w|)^{1-\frac1p+1-\frac1{p'}}dA(w)\right)\\
	&\quad+|f_n(0)| \|g\|_{H^\infty}\\
	&\lesssim\varepsilon\left(\|G^{H^2}_{g,z})\|^{p'}_{D^{p'}_{p'-1}}+ \sup_{n}\|f_n\|_{D^p_{p-1}}\right)\lesssim\e,\quad n\ge N.
	\end{split}
	\end{equation*}
Therefore $T_g:D^p_{p-1}\to H^\infty$ is compact by \cite[Lemma~3.6]{Tjani}.

Finally, let $X_p=H^p$. If $1<p<2$, then we may use the fact already proven that $T_g:\HL_p\to H^\infty$ is compact, and \eqref{Hardy1} to deduce the compactness of $T_g: H^p\to H^\infty$. In the case $2<p<\infty$ the same conclusion follows from  \eqref{Hardy2} and the compactness of $T_g:D^p_{p-1}\to H^\infty$. This finishes the proof of the theorem.
\hfill$\Box$

\section{Hardy, Hardy-Littlewood and Dirichlet-type spaces with $0<p\le1$}\label{Sec:2-3}

In this section we prove Theorems~\ref{Theorem:p=1} and \ref{Theorem:p=1'} in the said order. Since all the necessary auxiliary results are already stated in the previous sections, we can directly embark on the proofs.

\medskip

\noindent{\emph{Proof of}} Theorem~\ref{Theorem:p=1}.
(i). By \eqref{Hardy1} it suffices to show that $T(\HL_1,H^\infty)$ contains all polynomials. By Lemma~\ref{le:dualHL1} and \cite[Theorem~2.2]{CPPR}, we know that 
	$$
	\|T_g\|_{\HL_1\rightarrow H^{\infty}}\asymp\sup_{z\in\D}\|G^{H^2}_{g,z}\|_{\HL_{\infty}}. 
	$$
Since 
	$$
	\overline{G^{H^2}_{g,z}(w)}
	=\int_0^z\frac{g'(\z)}{1-\overline{w}\z}\,d\zeta
	=\sum_{k=0}^{\infty}\left(\sum_{n=0}^{\infty}(n+1)\widehat{g}(n+1)\frac{z^{n+k+1}}{n+k+1}\right)\overline{w}^k,
	$$
it is easy to show that 
	\begin{equation}\label{pillo}
	\sup_{z\in\D}\|G^{H^2}_{g,z}\|_{\HL_{\infty}}\lesssim
	\sup_{k\in\N\cup\{0\}}\left((k+1)\sum_{n=0}^{\infty}\frac{(n+1)|\widehat{g}(n+1)|}{n+k+1}\right).
	\end{equation}
If $g$ is a monomial, then $T_g:\HL_1\to H^\infty$ is bounded by \eqref{pillo}. It follows that $T(\HL_1,H^\infty)$ contains all polynomials.

(ii). Let $0<p<1$. Observe that
	$$
	\frac1{(1-\overline{w}\zeta)^{1+\frac1p}}=\sum_{k=0}^\infty c_p(k)(\overline{w}\zeta)^k,
	$$
where $c_p(k)\asymp(k+1)^\frac1p$ for all $k\in\N\cup\{0\}$. Hence
	$$
	\int_0^z\frac{g'(\zeta)}{(1-\overline{w}\zeta)^{1+\frac1p}}d\zeta
	=\sum_{n=0}^\infty\sum_{k=0}^\infty\frac{(n+1)\widehat{g}(n+1)c_p(k)}{n+k+1}\overline{w}^kz^{n+k+1}.
	$$
If $g$ is not a constant, then Lemma~\ref{lemma:dual-Dirichlet-2}, \cite[Theorem~2.2]{CPPR}, which can be applied to quasi-Banach spaces,  \eqref{Eq-alpha-bloch} and \eqref{Hardy1} yield
	\begin{equation*}
	\begin{split}
	\|T_g\|_{X_p\to H^\infty}
	&\gtrsim \|T_g\|_{D^p_{p-1}\to H^\infty}
	\asymp\sup_{z\in\D}\left(\left(\sup_{w\in\D}\left|\overline{\int_0^z\frac{g'(\zeta)\zeta}{(1-\overline{w}\zeta)^{2+\frac1p}}d\zeta}\right|(1-|w|)^2\right)+|g(z)-g(0)|\right)\\
	&\asymp\sup_{z\in\D}\left(\sup_{w\in\D}\left|\overline{\int_0^z\frac{g'(\zeta)}{(1-\overline{w}\zeta)^{1+\frac1p}}d\zeta}\right|(1-|w|)\right)\\
	&\gtrsim\sup_{z\in\D}\left|\int_0^z\frac{g'(\zeta)}{(1-\overline{z}\zeta)^{1+\frac1p}}d\zeta\right|(1-|z|)\\
	&=\sup_{z\in\D}\left|\sum_{n=0}^\infty\sum_{k=0}^\infty\frac{(n+1)\widehat{g}(n+1)c_p(k)}{n+k+1}|z|^{2k}z^{n+1}\right|(1-|z|)\\
	&\ge\sup_{0<r<1}\frac1{2\pi}\int_0^{2\pi}\left|\sum_{n=0}^\infty
	\left(\sum_{k=0}^\infty\frac{c_p(k)}{n+k+1}r^{2k+n+1}\right)(n+1)\widehat{g}(n+1)e^{i\theta(n+1)}\,d\theta\right|(1-r)\\
	&\gtrsim\sup_{0<r<1}\sum_{n=0}^\infty\frac{\sum_{k=0}^\infty\frac{c_p(k)}{n+k+1}r^{2k+n+1}(n+1)|\widehat{g}(n+1)|}{n+1}(1-r)\\
	&\gtrsim\limsup_{r\to1^-}\sum_{k=0}^\infty\frac{c_p(k)}{k+1}r^{2k}(1-r)
	\asymp\lim_{r\to1^-}\frac1{(1-r)^{\frac1p-1}}=\infty
	\end{split}
	\end{equation*}
because $0<p<1$. Therefore $T_g:X_p\rightarrow H^\infty$ is bounded if and only if $g$ is a constant.

(iii). By \eqref{Hardy1}, it suffices to consider the case of $X_1=D^1_0$, so assume that $T_g: D^1_0\rightarrow H^\infty$ is compact. Let $(H^{\infty})^\star$ denote the identification of the dual space of $H^{\infty}$ via the $A^2$-pairing. Then $T_g^*: (H^{\infty})^\star\rightarrow\B^2$ is compact by Lemma~\ref{lemma:dual-Dirichlet-2}. Let $K^{A^2}_{z}$ denote the reproducing kernel of the Hilbert space $A^2$, associated to the point $z\in\D$. Then $T_g^*(K^{A^2}_{z})=G^{A^2}_{g,z}$ for all $z\in\D$, and  
	\begin{equation*}
	\begin{split}
	\|K^{A^2}_{z}\|_{(H^{\infty})^\star}
	&=\sup_{\|f\|_{H^{\infty}}\le1}|\langle f, K^{A^2}_{z}\rangle|
	=\sup_{\|f\|_{H^{\infty}}\le1}\lim_{r\rightarrow1^-}\left|\sum_{n=0}^{\infty}\widehat{f}(n)\frac{z^n}{n+1}(n+1)r^n\right|\\
	&\le\sup_{\|f\|_{H^{\infty}}\le1}\|f\|_{H^{\infty}}
	\le1,\quad z\in\D.
	\end{split}
\end{equation*}
Therefore $\{G^{A^2}_{g,z}:z\in\D\}$ is relatively compact in $\B^2$. Hence, for given $\varepsilon>0$, there exist $z_1,\ldots,z_N\in\D$ such that for each $z\in\D$, we have $\|G^{A^2}_{g,z}-G^{A^2}_{g,z_j}\|_{\B^2}<\varepsilon$ for some $j=j(z)\in\{1,\ldots,N\}$. By using this and \eqref{Eq-alpha-bloch} we deduce
	\begin{equation*}
	\begin{split}
	&\sup_{a\in\D}\frac1{(1-|a|)^2}\int_{S(a)\setminus D(0,R)}|G^{A^2}_{g,z}(w)|^2(1-|w|)^2\,dA(w)\\
	&\quad\lesssim\|G^{A^2}_{g,z}-G^{A^2}_{g,z_j}\|_{H^\infty_1}^2+\sup_{a\in\D}\frac1{(1-|a|)^2}\int_{S(a)\setminus D(0,R)}|G^{A^2}_{g,z_j}(w)|^2(1-|w|)^2\,dA(w)\\
	&\quad \lesssim\e^2+\sup_{|a|\ge R}\frac1{(1-|a|)^2}\int_{S(a)}|G^{A^2}_{g,z_j}(w)|^2(1-|w|)^2\,dA(w).
	\end{split}
	\end{equation*}
Since $G^{A^2}_{g,z_j}\in\A\subset\B^2_0$ for each $j\in\{1,\ldots,N\}$, we obtain
	$$
	\lim_{R\rightarrow1^-}\sup_{a,z\in\D}\frac1{(1-|a|)^2}\int_{S(a)\setminus D(0,R)}|G^{A^2}_{g,z}(w)|^2(1-|w|)^2\,dA(w)=0,
	$$
which is equivalent to
	\begin{equation}\label{eq:compactdirichlet}
	\lim_{R\rightarrow1^-}\sup_{a,z\in\D}\int_{\D\setminus D(0,R)}|G^{A^2}_{g,z}(w)|^2(1-|\varphi_a(w)|^2)^2\,dA(w)=0
	\end{equation}
by the reasoning in the proof of \cite[Lemma 3.3]{G}, see \cite[Lemma 5.3]{PR} for further details. However, if $g$ is not a constant, then there exists an $N\in\N\cup\{0\}$ such that $\widehat{g}(N+1)\ne0$. Therefore, 
	\begin{equation*}
	\begin{split}
	&\sup_{a,z\in\D}\int_{\D\setminus D(0,R)}|G^{A^2}_{g,z}(w)|^2(1-|\varphi_a(w)|^2)^2\,dA(w)\\
	&\quad\ge\sup_{z\in\D}\int_{\D\setminus D(0,R)}|G^{A^2}_{g,z}(w)|^2(1-|\varphi_z(w)|^2)^2\,dA(w)\\
	&\quad=\sup_{z\in\D}(1-|z|)^2\left(\int_{\D\setminus D(0,R)}\left|\frac{G^{A^2}_{g,z}(w)}{(1-z\overline{w})^2}\right|^2
    (1-|w|)^2
    \,dA_2(w)\right)\\
	&\quad\asymp\sup_{z\in\D}\left((1-|z|)^2\int_{\D\setminus D(0,R)}
	\left|\left(\sum_{k=0}^{\infty}\left(\sum_{n=0}^{\infty}\frac{(n+1)\widehat{g}(n+1)(k+1)}{n+k+1}z^{n+k+1}\right)\overline{w}^k\right)
    \right.\right.\\
	&\qquad\quad\cdot\left.\left.
    \left(\sum_{j=0}^{\infty}(j+1)z^i\overline{w}^j\right)\right|^2
    (1-|w|)^2
    \,dA(w)\right)\\
	&\quad=\sup_{z\in\D}(1-|z|)^2
    \int_{\D\setminus D(0,R)}
	\left| \sum_{m=0}^{\infty} \left(\sum_{k=0}^{m}\left(\sum_{n=0}^{\infty}\frac{(n+1)\widehat{g}(n+1)(k+1)}{n+k+1}z^{n+k+1}\right)(m-k+1)z^{m-k}\right)
    \overline{w}^m\right|^2\\
	&\qquad\quad\cdot(1-|w|)^2\,dA(w)
    \\
	&\quad\asymp\sup_{z\in\D}(1-|z|)^2\Bigg(\sum_{m=0}^{\infty}\left|\sum_{k=0}^m\sum_{n=0}^{\infty}\frac{(n+1)\widehat{g}(n+1)(k+1)(m-k+1)}{n+k+1}z^{n+m+1}	\right|^2
    \int_R^1s^{2m+1}(1-s)^2\,ds\Bigg)\\
	&\quad\gtrsim\sup_{0<r<1}\Bigg((1-r)^2\Bigg(\sum_{m=0}^{\infty}r^{2m}\left(\int_R^1s^{2m+1}(1-s)^2\,ds\right)\\
	&\qquad\quad\cdot\int_0^{2\pi}\left|\sum_{n=0}^{\infty}(n+1)\widehat{g}(n+1)\left(\sum_{k=0}^{m}\frac{(k+1)(m-k+1)}{n+k+1}\right)r^{n+1}e^{i\theta(n+1)}\right|^2
    \Bigg)\,d\theta\\
	&\quad\gtrsim\sup_{0<r<1}|\widehat{g}(N+1)|^2(N+1)^2(1-r)^2r^{2N+2}\sum_{m=0}^{\infty}r^{2m}\left(\int_R^1s^{2m+1}(1-s)^2\,ds\right)\left(\sum_{k=0}^{m}(m-k+1)\right)^2\\
	&\quad\gtrsim R^{2N+1}(1-R)^2\int_R^1(1-s)^2\left(\sum_{m=0}^{\infty}(m+1)^4(Rs)^{2m+1}\right)\,ds\\
	&\quad\asymp R^{2N+1}(1-R)^2\int_R^1\frac{(1-s)^2}{(1-Rs)^5}ds\\
	&\quad\asymp R^{2N+1},\quad 0<R<1.
	\end{split}
	\end{equation*}
By letting $R\rightarrow1^-$ we obtain a contradiction with \eqref{eq:compactdirichlet}. Therefore $g$ must be a constant if $T_g: D^1_0\rightarrow H^\infty$ is compact. This finishes the proof of the theorem.
\hfill$\Box$

\bigskip

\noindent{\emph {Proof of}} Theorem~\ref{Theorem:p=1'}. By using the proof of Theorem~\ref{Theorem:p=1}(i) and standard arguments, we deduce 
	\begin{equation*}
	\begin{split}
	\|T_g\|_{D^1_0\to H^\infty}
	&\lesssim\|T_g\|_{H^1\to H^\infty}
	\lesssim\|T_g\|_{\HL_1\to H^\infty}
	\asymp\sup_{z\in\D}\|G^{H^2}_{g,z}\|_{\HL_{\infty}}\\
	&\asymp
	\sup_{k\in\N\cup\{0\}}\left((k+1)\sum_{n=0}^{\infty}\frac{(n+1)\widehat{g}(n+1)}{n+k+1}\right)
	\end{split}
	\end{equation*}
because $\widehat{g}(n)\geq0$ for all $n\in\mathbb{N}\cup\{0\}$ by the hypothesis. Thus $T_g:X_1\to H^\infty$ is bounded if \eqref{Eq:p=1} is satisfied.

Conversely, if $T_g:X_1\to H^\infty$ is bounded, then $T_g:D^1_0\to H^\infty$ is bounded by \eqref{Hardy1}. Therefore Lemma~\ref{lemma:dual-Dirichlet-2}, \cite[Theorem~2.2]{CPPR} and \eqref{Eq-alpha-bloch} yield
	\begin{equation*}
	\begin{split}
	\|T_g\|_{D^1_0\to H^\infty}
	&\asymp\sup_{z\in\D}\left(\left(\sup_{w\in\D}\left|\overline{\int_0^z\frac{g'(\zeta)\zeta}{(1-\overline{w}\zeta)^{3}}d\zeta}\right|(1-|w|)^2\right)+|g(z)-g(0)|\right)\\
	&\asymp\sup_{z\in\D}\left(\sup_{w\in\D}\left|\overline{\int_0^z\frac{g'(\zeta)}{(1-\overline{w}\zeta)^{2}}d\zeta}\right|(1-|w|)\right).
	\end{split}
	\end{equation*}
Now that $\widehat{g}(n)\geq0$ for all $n\in\mathbb{N}\cup\{0\}$ by the hypothesis, standard arguments yield
	\begin{equation}\label{mierda}
	\|T_g\|_{D^1_0\to H^\infty}
	\asymp\sup_{0\le s<1}\left(\left(\sum_{k=0}^{\infty}\sum_{n=0}^{\infty}\frac{(n+1)(k+1)\widehat{g}(n+1)}{(n+k+1)}s^{k}\right)(1-s)\right).
	\end{equation}
Since the coefficients
	$$
	\sum_{n=0}^{\infty}\frac{(n+1)(k+1)\widehat{g}(n+1)}{(n+k+1)}
	$$
are positive for all $k$, and increasing in $k$, we deduce
	\begin{equation*}
	\begin{split}
	\|T_g\|_{D^1_0\to H^\infty}
	&\gtrsim\sup_{K\in\N\cup\{0\}}\left(\left(\sum_{n=0}^{\infty}\frac{(n+1)(K+1)\widehat{g}(n+1)}{(n+K+1)}\right)\limsup_{s\to1^-}\sum_{k=K}^{\infty}s^{k}(1-s)\right)\\
	&=\sup_{K\in\N\cup\{0\}}\left((K+1)\sum_{n=0}^{\infty}\frac{(n+1)\widehat{g}(n+1)}{(n+K+1)}\right),
	\end{split}
	\end{equation*}
and thus \eqref{Eq:p=1} is satisfied. The norm estimate \eqref{eq:j11'} is an immediate consequence of the proof just established. 
This finishes the proof.
\hfill$\Box$

\section{$\BMOA$, Bloch space and $H^\infty_{\log}$}\label{Sec:Bloch}

This section is devoted to the proof of Theorem~\ref{Bounded}. Unlike the other main results, its proof does not require much tools, but can be carried out with relatively straightforward arguments.

\medskip

\noindent{\emph {Proof of}} Theorem~\ref{Bounded}.
The chain of inequalities \eqref{Eq:embedding} shows that (i)$\Rightarrow$(ii)$\Rightarrow$(iii) and
	\begin{equation}\label{Eq:1}
	\|T_g\|_{\BMOA\rightarrow H^\infty}\lesssim\|T_g\|_{X\rightarrow H^\infty}\lesssim \|T_g\|_{H_{\log}^\infty\rightarrow H^\infty}.
	\end{equation}
If (iii) is satisfied, then, by \cite[Theorem 2.5 (v)]{CPPR}, we have
	\begin{equation*}
	\begin{split}
	\|T_g\|_{\BMOA\rightarrow H^\infty}
	&\asymp\sup_{z\in\D}\|G^{H^2}_{g,z}\|_{H^1}
	=\sup_{z\in\D}\int_{0}^{2\pi}\left|\overline{\int_0^z\frac{g'(\z)}{1-\z e^{-i\theta}}\,d\zeta}\right|d\theta\\
	&=\sup_{z\in\D}\int_0^{2\pi}\left|\sum_{k=0}^\infty\sum_{n=0}^\infty\frac{(n+1)\overline{\widehat{g}(n+1)}\overline{z}^{n+k+1}}{n+k+1}e^{ik
\theta}\right|d\theta.
	\end{split}
	\end{equation*}
Since $\widehat{g}(n)\geq0$ for all $n\in\mathbb{N}\cup\{0\}$ by the hypothesis, Hardy's inequality \cite[p.~48]{D} and Fatou's lemma yield
	\begin{equation*}
	\begin{split}
	\|T_g\|_{\BMOA\rightarrow H^\infty}
	&\gtrsim\sup_{z\in\D}\sum_{k=0}^{\infty}\left|\sum_{n=0}^\infty\frac{(n+1)\widehat{g}(n+1)\overline{z}^{n+k+1}}{(n+k+1)(k+1)}\right|\\
	&\ge\sup_{0\le r<1}\sum_{k=0}^{\infty}\sum_{n=0}^\infty\frac{(n+1)\widehat{g}(n+1)r^{n+k+1}}{(n+k+1)(k+1)}\\
	&=\sum_{n=0}^\infty(n+1)\widehat{g}(n+1)\sum_{k=0}^{\infty}\frac{1}{(k+1)(n+k+1)}\\
	&\gtrsim\sum_{n=0}^\infty\widehat{g}(n+1)\log(n+2).
	\end{split}
	\end{equation*}
Therefore, (iii) implies (iv), and
	\begin{equation}\label{Eq:2}
	\|T_g\|_{\BMOA\rightarrow H^\infty}\gtrsim\sum_{n=0}^\infty\widehat{g}(n+1)\log(n+2).
	\end{equation}

If (iv) is satisfied, then 
	\begin{equation*}
	\begin{split}
	\int_0^1M_\infty(t,g')\log\frac{e}{1-t}\,dt 
	&=\int_0^1 \left(\sum_{n=0}^\infty (n+1)\widehat{g}(n+1)t^n \right)\log\frac{e}{1-t}\,dt\\
	&=\sum_{n=0}^\infty (n+1)\widehat{g}(n+1) \int_0^1  t^{n}\log\frac{e}{1-t}\,dt\\
	&\asymp\sum_{n=0}^\infty\widehat{g}(n+1)\log(n+2).
	\end{split}
	\end{equation*}
Thus (v) holds and
	\begin{equation}\label{Eq:3}
	\sum_{n=0}^\infty\widehat{g}(n+1)\log(n+2)
	\asymp\int_0^1M_\infty(t,g')\log\frac{e}{1-t}dt.
	\end{equation}
 
Finally, assume (v). Then
	\begin{equation*}
	\begin{split}
	\|T_g(f)\|_{H^\infty}
	=\sup_{z\in\D}\left|\int_0^zf(\z)g'(\z)d\z\right|
	&\le\sup_{0\le r<1}\int_0^rM_\infty(s,f)M_\infty(s,g')\,ds\\
	&\le\|f\|_{H_{\log}^\infty}\int_0^1M_{\infty}(s,g')\log\frac{e}{1-s}\,ds,\quad f\in\H(\D),
	\end{split}
	\end{equation*}
and therefore 
	\begin{equation}\label{Eq:4}
	\|T_g\|_{H_{\log}^\infty\rightarrow H^\infty}\lesssim\int_0^1M_\infty(r,g')\log\frac{e}{1-r}dr.
	\end{equation}
To complete the proof of the theorem it is now enough to show that $T_g: H^\infty_{\log}\to H^\infty$ is compact, that is
 $\lim_{n\to \infty}\|T_g(f_n)\|_{H^\infty}=0$ for each sequence $\{ f_n\}$ of analytic functions in $\D$ such that $\sup_n\|f_n\|_{H^\infty_{\log}}<\infty$
and $f_n\to0$ uniformly on compact subsets of $\D$.
Let $\e>0$. Fix $R=R(\e)\in(0,1)$ such that
	$$
	\int_R^1M_{\infty}(r,g')\log\frac{e}{1-r}dr<\e,
	$$
and then pick up $N=N(\e,R)\in\N$ such that $|f_n(z)|\le\e$ for all $ z \in\overline{D(0,R)}$ and $n\ge N$. Then, if $n\ge N$, we have
	\begin{equation*}
	\begin{split}
	\|T_g(f_n)\|_{H^\infty}
	&=\limsup_{|z|\to1^-}\left|\int_0^zf_n(\z)g'(\z)d\z\right|
	\le\limsup_{r\to1^-}\int_0^rM_\infty(s,f_n)M_\infty(s,g')\,ds\\
	&=\int_0^RM_\infty(s,f_n)M_\infty(s,g')\,ds
	+\int_R^1M_\infty(s,f_n)M_\infty(s,g')\,ds\\
	&\le\e\int_0^RM_\infty(s,g')\,ds
	+\|f_n\|_{H_{\log}^\infty}\int_R^1M_{\infty}(s,g')\log\frac{e}{1-s}\,ds
	\lesssim\e.
	\end{split}
	\end{equation*}
Thus $T_g:~H^\infty_{\log}\mapsto H^\infty$ is compact. The last thing to do is to observe that \eqref{Eq:1}--\eqref{Eq:4} imply the norm estimates \eqref{Eq:norm-estimates}. This finishes the proof.
\hfill$\Box$

\medskip

We finish the section and the paper by discussing briefly other natural choices for the space $X$ in Theorem~\ref{Bounded}. For a  nonnegative function  $\om\in L^1([0,1))$, the extension to $\D$, defined by $\om(z)=\om(|z|)$ for all $z\in\D$, is called a radial weight. For $0<p<\infty$ and such an $\omega$, the Lebesgue space $L^p_\om$ consists of complex-valued measurable functions $f$ on $\D$ such that
    $$
    \|f\|_{L^p_\omega}^p=\int_\D|f(z)|^p\omega(z)\,dA(z)<\infty.
    $$
The corresponding weighted Bergman space is $A^p_\om=L^p_\omega\cap\H(\D)$. For a radial weight $\om$, its associated weight $\om^\star$ is defined by $\om^\star(z)=\int_{|z|}^1 s \log\frac{s}{|z|}\om(s)\,ds$ for all $z\in\D\setminus\{0\}$. It arises naturally when the Hardy-Stein-Spencer formula is applied to the dilatation $f_r$ in order to establish the identity
	$$
	\|f\|_{A^p_\om}^p=\|\Delta|f|^p\|_{L^1_{\om^\star}}+\om(\D)|f(0)|^p,\quad f\in\H(\D),
	$$
see~\cite[Theorem~4.2]{PR} for details. Because the Laplacian of $|f|^p$ contains the factor $|f'|^2$, which can be interpreted as the Jacobian of the non-univalent change of variable $w=f(z)$, this equivalent norm is useful, for example, in the study of composition operators~\cite{PelRatSc}. But the associated weight comes to the picture also in some other instances which are more closely related to the topic of the present paper. To explain this, we say that a radial weight $\om$ belongs to the class~$\DD$ if there exists a constant $C=C(\om)\ge1$ such that $\widehat{\om}(r)\le C\widehat{\om}(\frac{1+r}{2})$ for all $0\le r<1$. Moreover, if there exist $K=K(\om)>1$ and $C=C(\om)>1$ such that $\widehat{\om}(r)\ge C\widehat{\om}\left(1-\frac{1-r}{K}\right)$ for all $0\le r<1$, then we write $\om\in\Dd$. The intersection $\DD\cap\Dd$ is denoted by $\DDD$. For $\om\in\DD$, the space $C^1(\om^\star)$ consists of $f\in\H(\D)$ such that the measure $|f'|^2\om^\star\,dA$ is a $1$-Carleson measure for $A^p_\om$~\cite[Theorem 6.1]{PelSum14}. As usual, we say that a positive Borel measure $\mu$ on $\D$ is a $p$-Carleson measure for $X$ if $X$ is continuously embedded into~$L^p_\mu$. The space $C^1(\om^\star)$ arises in the study of integration operators acting on weighted Bergman spaces. Indeed, it is known that, for each $0<p<\infty$, the operator $T_g$ is bounded from $A^p_\om$ into itself if and only if $g \in C^1(\om^\star)$~\cite[Theorem~6.4]{PelSum14}. For $\om\in\DDD$, the space $C^1(\om^\star)$ is nothing else but the Bloch space by the proof of \cite[Theorem~6.1(C)]{PelSum14}, but it may be a proper subspace of $\B$ by \cite[Theorem 6.1(D)]{PelSum14}, yet it always contains $\BMOA$. Therefore we may choose $X=C^1(\om^\star)$ in Theorem~\ref{Bounded}. It is worth observing that while $\BMOA$ and $\B$ are conformally invariant, there exists $\om\in\DD\setminus\DDD$ such that $C^1(\om^\star)$ is not \cite[Proposition~5.6]{PR}. Recall that a Banach space $X\subset\H(\D)$, equipped with a semi-norm $\rho_X$, is conformally invariant if there exists a constant $C=C(X)>0$ such that $\rho_X(f\circ\varphi) \le C \rho (f)_X$ for all $f\in X$ and for all automorphisms $\varphi$ of $\D$.

\end{document}